\documentclass[12pt, reqno]{amsart}
\usepackage{amssymb, amsmath, amsthm, cancel, hyperref, color, enumerate, verbatim, mathrsfs, mathtools, gensymb, caption, subcaption, graphicx, multicol, setspace, enumitem}
\usepackage[title]{appendix}

\usepackage[margin=1in]{geometry}

\newtheorem{thm}{Theorem}[section]
\newtheorem{lem}[thm]{Lemma}
\newtheorem{prop}[thm]{Proposition}
\newtheorem{cor}[thm]{Corollary}

\theoremstyle{definition}
\newtheorem{defn}[thm]{Definition}

\theoremstyle{remark}
\newtheorem*{rem}{Remark}

\newcommand{\CC}{\mathbb{C}}
\newcommand{\ZZ}{\mathbb{Z}}
\newcommand{\QQ}{\mathbb{Q}}
\newcommand{\NN}{\mathbb{N}}
\newcommand{\HH}{\mathbb{H}}

\newcommand{\SL}{\text{SL}_2(\ZZ)}
\newcommand{\mat}[4]{\begin{bmatrix} #1 & #2 \\ #3 & #4 \end{bmatrix}}

\title[Congruences for generalized cubic partitions via localization]{A congruence family modulo powers of $5$ for generalized cubic partitions via the localization method}
\author{Dalen Dockery}
\address{Department of Mathematics, University of Tennessee, Knoxville, TN 37996, USA}
\email{ddocker5@vols.utk.edu}
\begin{document}

\begin{abstract}
Recently Amdeberhan, Sellers, and Singh introduced a new infinite family of partition functions called generalized cubic partitions. Given a positive integer $d$, they let $a_d(n)$ be the counting function for partitions of $n$ in which the odd parts are unrestricted and the even parts are $d$-colored. These partitions are natural generalizations of Chan's notion of cubic partitions, as they coincide when $d=2.$ Many Ramanujan-like congruences exist in the literature for cubic partitions, and in their work Amdeberhan, Sellers, and Singh proved a collection of congruences satisfied by $a_d(n)$ for various $d \geq 1$, including an infinite family with prime moduli.  Our goal in this paper is to prove a family of congruences modulo powers of 5 for $a_3(n)$. More specifically, our main theorem asserts
\[a_3\left(5^{2\alpha}n +\gamma_{\alpha} \right) \equiv 0 \pmod{5^\alpha},\]
where
\[\gamma_{\alpha} \coloneqq 20 + \frac{19 \cdot 25 (25^{\alpha-1}-1)}{24}.\]
In order to prove these congruences, we use an approach centered around modular functions, as in the seminal work of Watson and Atkin on proving Ramanujan's congruences for the partition function $p(n)$. However, due to the complexity of the modular curve $X_0(10)$ associated to our modular functions, the classical method cannot be directly applied. Rather, we utilize the very recently developed localization method of Banerjee and Smoot, which is designed to treat congruence families over more complicated modular curves, such as $X_0(10).$ 
\end{abstract}

\maketitle

\section{Introduction}\label{sec:intro}
A partition of a natural number $n$ is a sum $n = \lambda_1 + \lambda_2 + \dots +\lambda_r$ of positive integers $\lambda_1 \geq \cdots \geq \lambda_r,$ called the parts of the partition, and we denote by $p(n)$ the number of partitions of $n$. In his celebrated 1919 paper (see \cite{BO}), Ramanujan revolutionized the study of the partition function $p(n)$ by discovering the infinite congruence families (correcting a slight mistake of Ramanujan in the exponent of the second congruence)
\begin{equation}
\begin{split}
p(5^jn+\beta_5(j)) &\equiv 0 \pmod{5^j}, \\
p(7^jn+\beta_7(j)) &\equiv 0 \pmod{7^{\lfloor j/2\rfloor+1}}  \label{eq:ram_cong}, \\ 
p(11^jn+\beta_{11}(j)) &\equiv 0 \pmod{11^j},
\end{split}
\end{equation}
where $24\beta_\ell(j) \equiv 1 \pmod{\ell^j}.$
Using clever $q$-series manipulations, Ramanujan was able to establish the $j=1$ cases of the first two congruences in \eqref{eq:ram_cong}, i.e. for the moduli $5$ and $7$. However, proofs of the general congruences modulo powers of 5 and 7 are attributed to Watson (\cite{Watson}). Watson's proof relied on special sequences $(L_{\alpha,\ell})_{\alpha \geq 1}$ of modular functions over the modular curves $X_0(\ell)$ ($\ell \in \{5,7\}$), as he showed that $L_{\alpha,\ell}$ naturally encodes the partition numbers $p\left(\ell^\alpha n + \beta_\ell(\alpha)\right).$ Thus, Watson was able to deduce the Ramanujan congruences modulo powers of $\ell$ by showing that $L_{\alpha,\ell}$ have integral Fourier coefficients that are all divisible by the appropriate power of $\ell$.

The Ramanujan congruences modulo powers of 11 evaded proof for many years, however; a proof of this congruence family was finally given by Atkin (\cite{Atkin}) nearly 30 years after Watson's work, using a suitably modified approach. The main challenge in this case lies in the fact that the modular curve $X_0(11)$ has genus 1. By contrast, $X_0(5)$ and $X_0(7)$ have genus 0. 

Since Watson's and Atkin's seminal work, the method of constructing sequences of modular functions to prove infinite congruence families has seen much success. In fact, many such proofs are very similar to those of Watson and Atkin, both in structure and in difficulty. However, others are notably harder, and recent work of Smoot (see \cite{Smoot2elong}, for example) suggests that another topological property of $X_0(N)$ plays a major role in the complexity of such a proof: the number of cusps. When the modular curve $X_0(N)$ has two cusps, such as in all three cases of the Ramanujan congruences for $p(n)$ or, more generally, when $N$ is prime, the spaces of modular functions with a single pole at a given cusp are quite easy to describe. In fact, for modular curves of genus 0, these spaces are isomorphic to $\mathbb{C}[t]$, where the modular function $t$ is classically called a Hauptmodul. 

When the cusp count exceeds 2, the modular functions may have poles at various cusps, and so representing such spaces of functions becomes much more difficult. One tool for tackling congruences over these modular curves is the so-called localization method, recently developed by Banerjee and Smoot (see \cite{BS}, \cite{BS2}, \cite{Smoot2elong}). Their key insight is to express the modular functions of interest not as polynomials in some given reference functions (such as a Hauptmodul), but rather as rational polynomials in the reference functions with predictable denominators. 

In this work, our goal is to use the localization method to prove an infinite congruence family modulo powers of 5 for the function $a_3(n).$ This partition function is a member of an infinite family recently introduced by Amdeberhan, Sellers, and Singh (\cite{ASS}) which they called generalized cubic partitions. Given an integer $d \geq 1$, Amdeberhan, Sellers, and Singh let $a_d(n)$ denote the number of partitions of $n$ in which the odd parts are unrestricted and the even parts may be independently chosen to be one of $d$ colors. This family of partitions naturally generalizes Chan's (\cite{Chan}) notion of cubic partitions $c(n)$, named after their connection to Ramanujan's continued cubic fraction. A cubic partition is one in which the odd parts are unrestricted and the even parts may be independently chosen to be one of 2 colors; thus, clearly $c(n) = a_2(n).$ Moreover, note that $a_1(n)=p(n)$, so generalized cubic partitions can also be thought of as a generalization of the classical partition function. For more background on (generalized) cubic partitions and their arithmetic properties, see papers such as \cite{Chan},\cite{CT}, \cite{CD}, \cite{ASS}, \cite{Guad}. 

We now state our main theorem, which provides a congruence family modulo powers of 5 for the function $a_3(n)$.

\begin{thm}\label{thm:main}
For all $k \geq 1,$ the congruence
\[a_3\left(5^{2\alpha}n +\gamma_{\alpha} \right) \equiv 0 \pmod{5^\alpha}\]
holds, where $\gamma_{\alpha}$ is the integer defined by
\begin{equation}\label{eq:gammas}
\gamma_{\alpha} \coloneqq 20 + \frac{19 \cdot 25 (25^{\alpha-1}-1)}{24}.
\end{equation}
\end{thm}

Theorem \ref{thm:main} is a corollary of a more technical theorem, Theorem \ref{thm:main2}; this theorem involves a special sequence of modular functions $(L_\alpha)_{\alpha \geq 1}$ that encodes $a_3(5^{2\alpha}n+\gamma_\alpha)$. Importantly, these modular functions of interest are over the modular curve $X_0(10),$ which is the same modular curve as is studied by Banerjee and Smoot in \cite{BS}. As such, our analysis here is quite convenient, as we are able to immediately apply several of their results, and many others follow analogously. Here our aim is to summarize arguments of Banerjee and Smoot in order to keep statements in context; however, for a more detailed account, we refer the reader to \cite{BS}. 

The remainder of this paper is organized as follows: in Section \ref{sec:background}, we present the required background on modular functions. In Section \ref{sec:seq}, we define the aforementioned sequence of modular functions $(L_\alpha)_{\alpha \geq 1}$ and present some helpful facts related to this sequence. In Section \ref{sec:prelim} we prove several preliminary results that are critical to our proof of Theorem \ref{thm:main2}, which we provide in Section \ref{sec:proof}. Finally, we close in Section \ref{sec:closing} with two isolated congruences satisfied by $a_d(n)$ to motivate further avenues of study. 

\section{Acknowledgments}
The author would like to sincerely thank Marie Jameson for her invaluable advice and guidance. 

\section{Background}\label{sec:background}
Denote by $\SL$ the group of $2 \times 2$ matrices over $\ZZ$ with determinant 1, i.e.
\[\SL \coloneqq \left\{\mat{a}{b}{c}{d} : ad-bc = 1, (a,b,c,d) \in \ZZ^4\right\}.\]
We are primarily concerned with the congruence subgroups $\Gamma_0(N) \leq \SL$, defined by
\[\Gamma_0(N) \coloneqq \left\{\mat{a}{b}{c}{d} \in \SL : c \equiv 0 \pmod{N}\right\}.\]

Let $\HH \coloneqq \{z \in \CC : \text{Im}(z) > 0\}$ be the complex upper half-plane. The group $\SL$ and its subgroups $\Gamma_0(N)$ act on $\HH$ via linear fractional transformations: for $\gamma = \mat{a}{b}{c}{d} \in \SL$ and $\tau \in \HH$, 
\[\gamma \tau \coloneqq \frac{a\tau+b}{c\tau+d}.\]
The orbits of $\mathbb{P}^1(\mathbb{Q}) \coloneqq \QQ \cup \{\infty\}$ under the action of $\Gamma_0(N)$ are called the cusps of $\Gamma_0(N)$. For example, $\SL = \Gamma_0(1)$ has one cusp, usually denoted by $\infty.$ 

\begin{defn}
The modular curve $X_0(N)$ is the set of orbits of $\hat{\mathbb{H}} \coloneqq \mathbb{H} \cup \mathbb{P}^1(\mathbb{Q})$ under the action of $\Gamma_0(N).$
\end{defn}

Classically, the modular curve $X_0(N)$ can be given the structure of a compact Riemann surface (see, e.g., \cite{DS}).  

\begin{defn}
Let $N \in \mathbb{Z}_{\geq 1}$. A holomorphic function $f : \HH \rightarrow \CC$ is called a modular function of level $\Gamma_0(N)$ if the following properties hold:
\begin{enumerate}
\item for all $\gamma \in \Gamma_0(N),$ $f(\tau)$ satisfies
\[f(\gamma \tau)=  f(\tau),\]
\item for all $\gamma = \mat{a}{b}{c}{d} \in \SL$ we have an expansion of $f(\gamma \tau)$ in the variable $q_N \coloneqq \exp(2\pi i \tau/N)$:
\[f(\gamma \tau) = \sum\limits_{n \geq n_\gamma} a_\gamma(n) q_N^{n \cdot \gcd(c^2,N)}.\]
\end{enumerate}
If $\gamma = I$ we simply write
\[f(\tau) = \sum\limits_{n \geq n_0} a(n)q^n\]
for $q \coloneqq \exp(2\pi i \tau),$ called the Fourier expansion of $f(\tau)$ (at $\infty$).
\end{defn}
If $n_\gamma \geq 0,$ we say that $f$ is holomorphic at the cusp $[a/c]_{N}.$ If $n_\gamma <0$, we say that $f$ has a pole of order $n_\gamma$ at the cusp $[a/c]_{N},$ and its principal part is 
\[\sum\limits_{n=n_\gamma}^{-1} a_\gamma(n) q_N^{n \cdot \gcd(c^2,N)}.\]
In either case, we denote the order of $f$ at the cusp $[a/c]_{N}$ by $\text{ord}_{a/c}^N(f) \coloneqq n_\gamma.$ Let $\mathcal{M}(\Gamma_0(N))$ be the $\CC$-vector space of all modular functions over $\Gamma_0(N)$. If $[a/c]_{N}$ is a cusp of $\Gamma_0(N),$ then the subspace of $\mathcal{M}(\Gamma_0(N))$ containing functions that are holomorphic at every cusp except possibly $[a/c]_N$ is denoted $\mathcal{M}^{a/c}(\Gamma_0(N))$. Henceforth, we shall identify a modular function $f(\tau)$ with its Fourier expansion $f(q)$. 

\begin{rem}
Meromorphic functions on $X_0(N)$ with poles supported at the cusps are in one-to-one correspondence with modular functions of level $\Gamma_0(N).$
\end{rem}

An extremely important operator on spaces of modular functions is Atkin's $U_d$-operator, defined by its action on Fourier expansions. 

\begin{defn}
For $d \in \NN$, define $U_d$ by 
\begin{equation}\label{eq:U}
U_d \left(\sum\limits_{n \geq n_0} a(n) q^n \right) \coloneqq \sum\limits_{dn \geq n_0} a(dn) q^n. 
\end{equation}
\end{defn}
Clearly $U_d$ is $\mathbb{C}$-linear. The following two lemmas record key mapping properties of the $U_d$-operator.

\begin{lem}[Lemmas 6 and 7 of \cite{AL}]\label{lem:Umap}
Suppose $d, N$ are positive integers such that $d \mid N$. Then 
\[U_d : \mathcal{M}(\Gamma_0(N)) \rightarrow \mathcal{M}(\Gamma_0(N)).\]
Moreover, if $d^2 \mid N$ then 
\[U_d : \mathcal{M}(\Gamma_0(N)) \rightarrow \mathcal{M}\left(\Gamma_0\left(\frac{N}{d}\right)\right).\]
\end{lem}

\begin{lem}[Lemma 7.1 of \cite{BS}]\label{lem:Umap2}
If $f \in \mathcal{M}^\infty(\Gamma_0(50)),$ then $U_5(f) \in \mathcal{M}^\infty(\Gamma_0(10)).$
\end{lem}

We shall also require the following well-known result that can be used to simplify the application of the $U_d$-operator on functions of a particular form. 

\begin{lem}\label{lem:U_PS}
Let $F(q)$ and $G(q)$ be $q$-series and take $d \in \NN.$ Then
\[U_d\left(F(q^d)G(q)\right) = F(q) \cdot U_d\left(G(q)\right).\]
\end{lem}

Many of the modular functions that are of interest to us are eta-quotients, which are functions of the form 
\[\prod\limits_{\delta \mid N} \eta(\delta \tau)^{r_\delta}\]
for integers $N, r_\delta.$ Here $\eta(\tau)$ is Dedekind's eta-function, defined by 
\[\eta(\tau) \coloneqq q^{1/24} \prod\limits_{n=1}^\infty (1-q^n).\]
The following special case of a well-known theorem of Newman allows one to easily prove that a given eta-quotient is a modular function. 

\begin{lem}[Theorem 1.64 of \cite{OnoWeb}]\label{lem:eta-quotient1}
Suppose that the eta-quotient $\prod\limits_{\delta \mid N} \eta(\delta \tau)^{r_\delta}$ satisfies
\begin{enumerate}
\item $\sum\limits_{\delta \mid N} r_{\delta} = 0,$
\item $\sum\limits_{\delta \mid N} \delta r_\delta \equiv 0 \pmod{24},$
\item $\sum\limits_{\delta \mid N} \frac{N}{\delta} r_\delta \equiv 0 \pmod{24},$
\item $\prod\limits_{\delta \mid N} \delta^{r_\delta}$  is a perfect square in $\mathbb{Q}$. 
\end{enumerate}
Then $\prod\limits_{\delta \mid N} \eta(\delta \tau)^{r_\delta} \in \mathcal{M}(\Gamma_0(N)).$ 
\end{lem}
Furthermore, the order of vanishing of an eta-quotient at a cusp of $\Gamma_0(N)$ is straightforward to compute, thanks to the following result typically credited to Ligozat. 

\begin{lem}[Theorem 1.65 of \cite{OnoWeb}]\label{lem:eta-quotient2}
The order of vanishing of the eta-quotient 
$\prod\limits_{\delta \mid N} \eta(\delta \tau)^{r_\delta} \in \mathcal{M}(\Gamma_0(N))$ at the cusp $[a/c]_N$ of $\Gamma_0(N)$ is given by
\[\frac{N}{24} \sum\limits_{\delta \mid N} \frac{\gcd(c,\delta)^2r_\delta}{\gcd(c,\frac{N}{c})c\delta}.\]
\end{lem}

For notational convenience, we define $f_r$ for $r \geq 1$ by the infinite product
\[f_r \coloneqq \prod\limits_{n=1}^\infty (1-q^{rn}).\]
The following lemma provides lower bounds on the orders of vanishing of certain modular functions.

\begin{lem}[Theorem 39 of \cite{Radu}]\label{lem:oov_bound}
Let $M$ be a positive integer and define integers $b(n)$ by 
\[\sum\limits_{n=0}^\infty b(n) q^n \coloneqq \prod\limits_{\delta \mid M} f_{\delta}^{r_\delta}.\]
Suppose that $t,m,N$ are non-negative integers with $0 \leq t < m$ and that
\[g \coloneqq \prod\limits_{\lambda \mid N} \eta(\lambda \tau)^{s_{\lambda}} \sum\limits_{n=0}^\infty b(mn+t) q^n \in \mathcal{M}(\Gamma_0(N)).\]
Then 
\[\textup{ord}_{a/c}^N(g) \geq \frac{N}{24\gcd(c^2,N)}\left(\min\limits_{0 \leq \ell \leq m-1} \sum\limits_{\delta \mid M} r_\delta \frac{\gcd(\delta(a+\ell c \cdot \gcd(m^2-1,24)),mc)^2}{\delta m}+\sum\limits_{\lambda \mid N} \frac{s_\lambda \gcd(\lambda,c)^2}{\lambda}\right).\]
\end{lem}

\section{Sequence of Modular Functions}\label{sec:seq}
Here and throughout, we follow the notation and technique of Banerjee and Smoot (\cite{BS}). As mentioned in Section \ref{sec:intro}, our proof of Theorem \ref{thm:main} critically relies on a sequence of modular functions $(L_\alpha)_{\alpha \geq 1}$ over $\Gamma_0(10)$. Our first task is to define this sequence of modular functions; to this end, set
\[\Phi \coloneqq \frac{\eta(25\tau)\eta(50\tau)^2}{\eta(\tau)\eta(2\tau)^2} = q^5 \frac{f_{25}f_{50}^2}{f_1f_2^2} \in \mathcal{M}(\Gamma_0(50)).\]
Define the operators $U^{(0)}, U^{(1)}$ by 
\begin{align*} 
U^{(0)} (f) &\coloneqq U_5(\Phi \cdot f), \\
U^{(1)} (f) &\coloneqq U_5(f),
\end{align*} 
where $U_5$ is as in \eqref{eq:U}. Now put $L_0 \coloneqq 1$, and for $n \geq 0$ let
\begin{equation}\label{eq:Ldef}
L_{\alpha+1} \coloneqq \begin{cases}  U^{(0)}(L_{\alpha}), & \alpha \text{ even},  \\ U^{(1)}(L_{\alpha}), & \alpha \text{ odd}. \end{cases}
\end{equation}
We now show that this sequence naturally contains values of $a_3(n)$ in the arithmetic progressions of interest. Note first that the generating function for $a_3(n)$ is given by 
\[\sum\limits_{n=0}^\infty a_3(n) q^n = \frac{1}{f_1 f_2^2}.\]

\begin{lem}\label{lem:L_seq}
For all $\alpha \geq 1$,
\[L_{2\alpha} = qf_1f_2^2 \sum\limits_{n=0}^\infty a_3\left(5^{2\alpha} n + \gamma_{\alpha}\right)q^n.\]
\end{lem}
\begin{proof} 
Proceed by induction. Note that 
\begin{align*}
L_1 &= U^{(0)}(1) = U_5(\Phi) \\
&= U_5\left(q^5 \frac{f_{25}f_{50}^2}{f_1f_2^2}\right) = qf_5f_{10}^2 U_5\left(\sum\limits_{n=0}^\infty a_3(n)q^n\right) \\
&= qf_5f_{10}^2 \sum\limits_{n=0}^\infty a_3(5n)q^n,
\end{align*}
thanks to Lemma \ref{lem:U_PS}. Thus
\begin{align*}
L_2 &= U^{(1)}(L_1) = U_5(L_1) \\
&=U_5\left(f_5 f_{10}^2 \sum\limits_{n=0}^\infty a_3(5n) q^{n+1}\right) \\
&= q f_1 f_2^2 \sum\limits_{n=0}^\infty a_3(25n+20) q^n.
\end{align*}
Now suppose that the conclusion holds for some $\alpha \geq 1.$ Then
\begin{align*}
L_{2\alpha+1} &= U^{(0)}(L_{2\alpha}) = U_5\left(q^6f_{25}f_{50}^2 \sum\limits_{n=0}^\infty a_3\left(5^{2\alpha} n + \gamma_{\alpha}\right)q^n\right) \\
&= qf_5f_{10}^2 U_5\left(\sum\limits_{n=0}^\infty a_3\left(5^{2\alpha}n+\gamma_\alpha\right)q^{n+1}\right) \\
&= qf_5f_{10}^2 \sum\limits_{n=1}^\infty a_3\left(5^{2\alpha+1}n - 5^{2\alpha}+\gamma_\alpha\right) q^n, \\
\end{align*}
again thanks to Lemma \ref{lem:U_PS}. Hence
\begin{align*} 
L_{2(\alpha+1)} &= L^{(1)}(L_{2\alpha+1}) = U_5(L_{2\alpha+1}) \\
&=U_5\left(qf_5f_{10}^2 \sum\limits_{n=1}^\infty a_3\left(5^{2\alpha+1}n - 5^{2\alpha}+\gamma_\alpha\right) q^n\right) \\
&= f_1f_2^2 U_5\left(\sum\limits_{n=2}^\infty a_3\left(5^{2\alpha+1}n - 5^{2\alpha+1} - 5^{2\alpha}+\gamma_\alpha\right) q^n \right) \\
&= qf_1 f_2^2 \sum\limits_{n=0}^\infty a_3\left(5^{2(\alpha+1)}n + 5^{2\alpha+2} - 5^{2\alpha+1} - 5^{2\alpha} + \gamma_\alpha\right)q^n \\
&= qf_1f_2^2 \sum\limits_{n=0}^\infty a_3 \left(5^{2(\alpha+1)}n + \gamma_{\alpha+1}\right) q^n,
\end{align*}
as
\begin{align*}
5^{2\alpha+2} - 5^{2\alpha+1} - 5^{2\alpha} + \gamma_\alpha &=  19 \cdot 5^{2\alpha} + \left(20 + \frac{19 \cdot 25(25^{\alpha-1}-1)}{24}\right) \\
&= 20 + \frac{19 \cdot 25(25^{\alpha}-1)}{24} \\
&= \gamma_{\alpha+1}.
\end{align*}
\end{proof}

By Lemmas \ref{lem:eta-quotient1} and \ref{lem:eta-quotient2} we quickly verify that $\Phi \in \mathcal{M}(\Gamma_0(50)),$ and so by Lemma \ref{lem:Umap} we know $L_1 = U_5(\Phi) \in \mathcal{M}(\Gamma_0(10)).$ Using Lemma \ref{lem:oov_bound}, we compute
\begin{align*}
&\text{ord}_\infty^{10}(L_1) \geq 1,\\
&\text{ord}_{1/5}^{10}(L_1) \geq 1,\\
&\text{ord}_{1/2}^{10}(L_1) \geq -5,\\
&\text{ord}_0^{10}(L_1) \geq -4.\\
\end{align*}
Now let 
\[z = z(\tau) \coloneqq \frac{\eta(2\tau)^5\eta(5\tau)}{\eta(\tau)^5\eta(10\tau)} = \frac{f_2^5 f_5}{f_1^5f_{10}}.\]
Again using Lemmas \ref{lem:eta-quotient1} and \ref{lem:eta-quotient2}, we see that $z$ is a Hauptmodul at $[0]_{10}$ with positive order at $[1/2]_{10}.$ That is,
\begin{align*}
&\text{ord}_\infty^{10}(z) = 0,\\
&\text{ord}_{1/5}^{10}(z) = 0,\\
&\text{ord}_{1/2}^{10}(z) = 1,\\
&\text{ord}_0^{10}(z) = -1.\\
\end{align*}
So, $z^5 L_1$ is a modular function over $\Gamma_0(10)$ whose only pole is at $[0]_{10}$, and thus we may write $z^5L_1$ as a polynomial in $z$. As in \cite{BS}, however, the polynomial obtained in this way has rational coefficients with denominators that are divisible by large powers of 5:
\[z^5L_1 = -\frac{1}{625} - \frac{6}{625} z + \dots.\]
Thus, we shift our attention to the function 
\[x = x(\tau) \coloneqq \frac{\eta(2\tau)\eta(10\tau)^3}{\eta(\tau)^3\eta(5\tau)} = q \frac{f_2f_{10}^3}{f_1^3f_5},\]
which is closely related to $z$ as shown by the following identity. 

\begin{lem}[Lemma 4.2 of \cite{BS}]
We have 
\[z = 1+5x.\]
\end{lem}

It will turn out that our modular functions $L_\alpha$ are not expressible as polynomials in $x$, but rather they are expressible as rational polynomials in $x$ whose denominators are powers of $1+5x$. For example,

\begin{equation}\label{eq:L1}
\begin{split}
L_1 = \frac{1}{(1+5x)^5} (x&+40x^2+794x^3+9125x^4+64475x^5+286000x^6\\
&+7800000x^7+1200000x^8+800000x^9).
\end{split}
\end{equation}
More generally, we have the following.

\begin{thm}\label{thm:main2}
Let $\alpha \geq 1,$ and set
\[\psi \coloneqq \psi(\alpha) = \left\lfloor \frac{5^{\alpha+2}}{24}\right\rfloor +1-\gcd(\alpha,2).\]
Then 
\[\frac{(1+5x)^\psi}{5^{\lfloor \alpha/2 \rfloor}} L_{\alpha} \in \mathbb{Z}[x].\]
\end{thm}
Since $1+5x \equiv 1 \pmod{5},$ the presence of $(1+5x)^\psi$ does not affect the 5-divisibility of $L_\alpha.$ Thus, coupling Theorem \ref{thm:main2} with Lemma \ref{lem:L_seq} provides an immediate proof of \ref{thm:main}, and so for the remainder of this paper, we focus our attention on proving Theorem \ref{thm:main2}. 

\section{Preliminary Results}\label{sec:prelim}
We begin by defining the spaces of rational polynomials in which the members of our sequence $(L_\alpha)_{\alpha \geq 1}$ lie. Set

\begin{align*}
\theta_0(m) &\coloneqq \begin{cases} 0, & 1 \leq m \leq 4, \\ \left\lfloor \frac{5m-1}{7} \right\rfloor -2 , & m \geq 5, \end{cases} \\
\theta_1(m) &\coloneqq \begin{cases} 0, & 1 \leq m \leq 7, \\ \left\lfloor \frac{5m-2}{7} \right\rfloor -5, & m \geq 8. \end{cases}
\end{align*}

\begin{defn} Let $s : \mathbb{Z}_{\geq 1} \rightarrow \mathbb{Z}$ be an arbitrary discrete function. Then for $n \in \mathbb{Z}$ we set 
\begin{align*}
\mathcal{V}_n^{(0)} &\coloneqq \left\{\frac{1}{(1+5x)^n} \sum\limits_{m \geq 1} s(m) \cdot 5^{\theta_0(m)} \cdot x^m \right\}, \\
\hat{\mathcal{V}}_n &\coloneqq \left\{\frac{1}{(1+5x)^n} \sum\limits_{m \geq 1} s(m) \cdot 5^{\theta_1(m)} \cdot x^m \right\}, 
\end{align*}
\begin{equation}\label{eq:Vn1}
\mathcal{V}_n^{(1)} \coloneqq \left\{f \in \hat{\mathcal{V}}_n : \begin{pmatrix} s(1) +  s(2) + s(3) + 2s(4) + s(5) \\ 4s(4) + s(6) + s(7) + s(8) \end{pmatrix} \equiv \begin{pmatrix} 0 \\ 0 \end{pmatrix} \pmod{5} \right\}.
\end{equation}
\end{defn}

Next, we recall useful modular equations over $\Gamma_0(N)$ that are satisfied by $x$ and $z.$ To state these identities, define 
\begin{align*}
&a_0(\tau) \coloneqq -(x+20x^2+150x^3+500x^4+625x^5), \\
&a_1(\tau) \coloneqq -(15x+305x^2+2325x^3+7875x^4+10000x^5,\\
&a_2(\tau) \coloneqq  -(85x+1750x^2+13525x^3+46500x^4+60000x^5),\\
&a_3(\tau) \coloneqq -(215x+4475x^2+35000x^3+122000x^4+160000x^5), \\
&a_4(\tau) \coloneqq -(205x+4300x^2+34000x^3+120000x^4+160000x^5),\\
\end{align*}
and 
\begin{align*}
&b_0(\tau) \coloneqq -z^5, \\
&b_1(\tau) \coloneqq 1+5z+5z^2+5z^3+5z^4-16z^5, \\
&b_2(\tau) \coloneqq -4-15z+10z^2+35z^3+60z^4-96z^5, \\
&b_3(\tau) \coloneqq 6+15z-35^2+40z^3+240z^4-256z^5, \\
&b_4(\tau) \coloneqq -4-5z+20z^2-80z^3+320z^4-256z^5, \\
&b_5(\tau) \coloneqq 1. \\
\end{align*}

\begin{lem}[Theorems 4.3 and 4.4 of \cite{BS}]\label{lem:mod_eq} We have the relations (correcting the coefficient of $x^3$ in $a_2(\tau)$ and the the leading term of the second equation)
\begin{equation}\label{eq:xmodeq}
x^5+\sum\limits_{j=0}^4 a_j(5\tau)x^j = 0
\end{equation}
and 
\begin{equation}\label{eq:zmodeq}
z^5+\sum\limits_{k=0}^4 b_k(5\tau) z^k = 0.
\end{equation}
\end{lem}

In light of the form of $L_1$ given above in \eqref{eq:L1} and the way in which subsequent functions $L_\alpha$ are constructed as defined by \eqref{eq:Ldef}, we aim to understand the functions $U^{(i)}(x^m/ (1+5x)^n),$ where $i \in \{0,1\}.$ 
Using the modular equations \eqref{eq:xmodeq} and \eqref{eq:zmodeq}, we are able to construct recurrence relations for these functions of interest.

\begin{lem}\label{lem:recurrence}
For $m, n \in \ZZ$ and $i \in \{0,1\},$ we have 
\[U^{(i)}\left(\frac{x^m}{ (1+5x)^n}\right) = - \frac{1}{(1+5x)^5} \sum\limits_{j=0}^4 \sum\limits_{k=1}^5 a_j(\tau)b_k(\tau) \cdot U^{(i)}\left(\frac{x^{m+j-5}}{(1+5x)^{n-k}}\right).\]
\begin{proof}
The proof follows, mutatis mutandis, as in the proof of Lemma 5.1 of \cite{BS}. 
\end{proof}
\end{lem}

Our next goal is to use Lemma \ref{lem:recurrence} to prove explicit formulas for $U^{(i)}(x^m/ (1+5x)^n)$ as rational polynomials in $x$ whose numerators have coefficients with predicable 5-adic valuations, as seen in the following theorem. 

\begin{defn}
For $m,r,\ell \in \ZZ$, let
\begin{align*}
\pi_0(m,r) &\coloneqq \max\left(0, \left\lfloor \frac{5r-m+2}{7} \right\rfloor -5 \right), \\
\pi_1(m,r) &\coloneqq \left\lfloor \frac{5r-m}{7} \right\rfloor.
\end{align*}
\end{defn}

Recall that a function $h : \ZZ^n \rightarrow \ZZ$ is called a discrete array if $h(m_1, m_2, \dots, m_n)$ has finite support as a function of $m_n$ for fixed $m_1, m_2, \dots, m_{n-1}$.

\begin{thm}\label{thm:h's}
There are discrete arrays $h_0, h_1 : \mathbb{Z}^3 \rightarrow \mathbb{Z}$ such that 
\begin{align*} 
U^{(0)}\left(\frac{x^m}{(1+5x)^n}\right) &= \frac{1}{(1+5x)^{5n+5}} \sum\limits_{r \geq \lceil (m+4)/5 \rceil} h_0(m,n,r) \cdot 5^{\pi_0(m,r)} \cdot x^r, \\
U^{(1)}\left(\frac{x^m}{(1+5x)^n}\right) &= \frac{1}{(1+5x)^n} \sum\limits_{r \geq \lceil m/5 \rceil} h_1(m,n,r) \cdot 5^{\pi_1(m,r)} \cdot x^r,
\end{align*}
for all $m, n \geq 0.$ 
\begin{proof}
The identity for $i=1$ is proved as part of Theorem 5.3 of \cite{BS}, and for $i=0$ we proceed by induction on $m$ and $n$. First, we consider $0 \leq m,n \leq 4.$ However, we need not explicitly prove all 25 of these base cases, as they can be constructed algorithmically from only five such equations. Indeed, by the Binomial Theorem 
\begin{align*} 
U^{(0)}\left(\frac{x^m}{(1+5x)^n}\right) &= \frac{1}{5^m} U^{(0)}\left(\frac{(z-1)^m}{z^n}\right) \\
&= \frac{1}{5^m} \sum\limits_{r=0}^n (-1)^{m-r} {m \choose r} \cdot U^{(0)}(z^{r-n}).
\end{align*}
So, if we are able to compute $U^{(0)}(z^{n})$ for all integers $n$, then we can quickly compute our 25 initial relations. Taking the modular equation \eqref{eq:zmodeq} and multiplying both sides by $z^n/b_0(5\tau)$, we see
\[z^n = -\frac{1}{b_0(5\tau)} \sum\limits_{k=1}^5 b_k(5\tau) z^{n+k}.\]
Thus,
\begin{align}
U^{(0)}(z^n) &= - \sum\limits_{k=1}^5 U^{(0)} \left(\frac{b_k(5\tau)z^{n+k}}{b_0(5\tau)}\right) \nonumber \\
&= -\sum\limits_{k=1}^5 \frac{b_k(\tau)}{b_0(\tau)} U^{(0)}(z^{n+k}) \label{eq:z_recurr}
\end{align}
thanks to Lemma \ref{lem:U_PS}. Since $n+k > n$ for $1 \leq k \leq 5,$ we only need to consider $U^{(0)}(z^n)$ for positive powers $n$ by inductively applying \eqref{eq:z_recurr}. But for $n \geq 1$,
\begin{align*} 
U^{(0)}(z^n) &= U^{(0)}((1+5x)^n) = U^{(0)}\left(\sum\limits_{i=0}^n {n \choose i} (5x)^r\right) \\
&= \sum\limits_{r=0}^n {n \choose r} 5^r \cdot U^{(0)}(x^r).
\end{align*}
Finally, note that for any $n \geq 1$ we may express $U^{(0)}(x^n)$ in terms of $U^{(0)}(x^k)$ for $0 \leq k \leq 4$ by applying \eqref{eq:xmodeq} as necessary. In Appendix \ref{app}, we provide explicit expressions for these five functions as rational polynomials in $x$ with denominator $(1+5x)^5,$ from which the remaining base cases may be quickly settled as describe above. 

In order to prove these identities, consider the function $\Phi(\tau) x(\tau)^k x(5\tau)^{-29} z(5\tau)^5$. Using Lemma \ref{lem:eta-quotient2} we may calculate the orders of each constituent factor at the cusps of $\Gamma_0(50)$; these values are listed in Table \ref{tab:cusps}, from which we immediately conclude
\[\Phi(\tau) x(\tau)^k x(5\tau)^{-29} z(5\tau)^5 \in \mathcal{M}^\infty(\Gamma_0(50))\]
since $0 \leq k \leq 4.$ 
\begin{table}
\begin{center}
\begin{tabular}{|c||c|c|c|c||c|}
\hline
Cusp & $\Phi(\tau)$ & $x(\tau)$ & $x(5\tau)$ & $z(5\tau)$ & $\Phi(\tau) x(\tau)^k x(5\tau)^{-29} z(5\tau)^5$ \\ \hline
$[\infty]_{10}$ & 5 & 1 & 5 & 0 & $k-140$ \\ 
$[1/25]_{10}$ & 4 & 0 & 0 & 0 & 4 \\
$[1/10]_{10}$ & 0 & 1 & 0 & 1 & $k+5$ \\
$[1/5]_{10}$ & 0 & 0 & $-1$ & $-1$ & 4 \\
$[3/10]_{10}$ & 0 & 1 & 0 & 1 & $k+5$ \\
$[2/5]_{10}$ & 0 & 0 & $-1$ & $-1$ & 4 \\
$[1/2]_{10}$ & -5 & 0 & 0 & 1 & 0 \\
$[3/5]_{10}$ & 0 & 0 & $-1$ & $-1$ & 4 \\
$[7/10]_{10}$ & 0 & 1 & 0 & 1 & $k+5$ \\
$[4/5]_{10}$ & 0 & 0 & $-1$ & $-1$ & 4 \\
$[9/10]_{10}$ & 0 & 1 & 0 & 1 & $k+5$ \\
$[0]_{10}$ & -4 & -5 & $-1$ & $-1$ & $20-5k$ \\ \hline
\end{tabular}
\caption{\label{tab:cusps}Modular cusp analysis over $X_0(50)$}
\end{center}
\end{table}
As such, 
\[U_5\left(\Phi(\tau) x(\tau)^k x(5\tau)^{-29} z(5\tau)^5\right) \in \mathcal{M}^\infty(\Gamma_0(10)),\] 
thanks to Lemma \ref{lem:Umap2}. Equivalently, by Lemma \ref{lem:U_PS}
\begin{equation}\label{eq:modfx1}
\frac{z^5}{x^{29}} U^{(0)}(x^k) \in \mathcal{M}^\infty(\Gamma_0(10)).
\end{equation}
On the other hand, each of the identities in Appendix \ref{app} is of the form
\[(1+5x)^5 U^{(0)}(x^k) = p_{k}(x)\]
for some polynomial $p_k(x) \in \mathbb{Z}[x]$. Note that for each $0 \leq k \leq 4,$ 
\begin{equation}\label{eq:modfx2}
\frac{p_k(x)}{x^{29}} \in \mathbb{Z}[x^{-1}] \subseteq \mathcal{M}^\infty(\Gamma_0(10)).
\end{equation}
Computing the Fourier expansions of the modular functions in \eqref{eq:modfx1} and \eqref{eq:modfx2}, we see that they have the same principal part, hence their difference must be a modular function over $\Gamma_0(10)$ with no poles. The only such functions are constants, and since the Fourier expansions of these functions have the same constant term, their difference must be 0. This proves the five initial relations in Appendix \ref{app}, thereby settling the base cases $0 \leq m, n \leq 4$ of Theorem \ref{thm:h's}. 

The inductive step follows identically to that of Theorem 5.3 of \cite{BS}.
\end{proof}
\end{thm}

The discrete arrays $h_i(m,n,r)$ appearing in Theorem \ref{thm:h's} satisfy internal congruences in $n$ modulo 5, which play a pivotal role in the proof of Theorem \ref{thm:main2}. 

\begin{prop}\label{prop:h_cong}
For all $m,n,r \geq 1$ and $i \in \{0,1\},$ we have
\[h_i(m,n,r) \equiv h_i(m,n-5,r) \pmod{5}.\]
\begin{proof}
The proof is identical to that of Theorem 5.5 of \cite{BS}.
\end{proof}
\end{prop}

By explicitly computing $h_1(m,n,1)$ and $h_1(m,n,2)$ for $1 \leq m \leq 8$ and $0 \leq n \leq 4$, we immediately obtain the following congruences. 

\begin{cor}\label{cor:h_cor}
For all $n \in \ZZ$, 
\begin{align*}
&\begin{pmatrix} h_1(1,n,1) & h_1(2,n,1) & h_1(3,n,1) & h_1(4,n,1) & h_1(5,n,1) \\ h_1(4,n,2) & h_1(5,n,2) & h_1(6,n,2) & h_1(7,n,2) & h_1(8,n,2) \end{pmatrix} \\
&\equiv \begin{pmatrix} 1 & 1 & 1 & 2 & 1 \\ 4 & 0 & 1 & 1 & 1 \end{pmatrix} \pmod{5}.
\end{align*}
\end{cor}

\section{Proof of Theorem \ref{thm:main2}}\label{sec:proof}
We first consider the application of the $U^{(0)}$-operator to functions in $\mathcal{V}_{n}^{(0)}.$

\begin{thm}\label{thm:U0_Vhat}
Suppose $f \in \mathcal{V}_{n}^{(0)}.$ Then $U^{(0)}(f) \in \hat{\mathcal{V}}_{5n+5}.$
\begin{proof}
Since $f \in \mathcal{V}_n^{(0)}$, we can write
\[f = \frac{1}{(1+5x)^n} \sum\limits_{m \geq 1} s(m) \cdot 5^{\theta_0(m)} \cdot x^m.\]
From Theorem \ref{thm:h's},
\begin{align*}
U^{(0)}(f) &= \sum\limits_{m \geq 1} s(m) \cdot 5^{\theta_0(m)} \cdot U^{(0)} \left( \frac{x^m}{(1+5x)^n}\right) \\
&= \frac{1}{(1+5x)^{5n+5}} \sum\limits_{m \geq 1} \sum\limits_{r \geq \lceil (m+4)/5 \rceil} s(m) \cdot h_0(m,n,r) \cdot 5^{\theta_0(m) + \pi_0(m,r)} \cdot x^r \\
&= \frac{1}{(1+5x)^{5n+5}} \sum\limits_{r \geq 1} \sum\limits_{m \geq 1}  s(m) \cdot h_0(m,n,r) \cdot 5^{\theta_0(m) + \pi_0(m,r)} \cdot x^r.
\end{align*}
Therefore, it suffices to show that, for all $m\leq 5r-4,$
\[\theta_0(m) + \pi_0(m,r) \geq \theta_1(r).\]
This inequality is proved in Theorem 6.1 of \cite{BS}. 
\end{proof}
\end{thm}

Next, we show that $U^{(1)}$ maps functions in $\mathcal{V}_n^{(1)}$ to $\mathcal{V}_{n'}^{(0)}$ while picking up an extra power of 5. 

\begin{thm}\label{thm:U1_V}
Suppose $f \in \mathcal{V}_n^{(1)}.$ Then
\[\frac{1}{5} U^{(1)}(f) \in \mathcal{V}_{5n}^{(0)}.\]
\begin{proof}
This is Theorem 6.2 of \cite{BS} with the assumption $n \equiv 1 \pmod{5}$ removed; we summarize the proof here for completeness. 

Since $f \in \mathcal{V}_n^{(1)},$ we can write 
\[f = \frac{1}{(1+5x)^n} \sum\limits_{m \geq 1} s(m) \cdot 5^{\theta_1(m)} \cdot x^m.\]
From Theorem \ref{thm:h's},
\begin{align*}
U^{(1)}(f) &= \sum\limits_{m \geq 1} s(m) \cdot 5^{\theta_1(m)} \cdot U^{(1)} \left( \frac{x^m}{(1+5x)^n}\right) \\
&= \frac{1}{(1+5x)^{5n}} \sum\limits_{m \geq 1} \sum\limits_{r \geq \lceil m/5 \rceil} s(m) \cdot h_1(m,n,r) \cdot 5^{\theta_1(m) + \pi_1(m,r)} \cdot x^r \\
&= \frac{1}{(1+5x)^{5n+5}} \sum\limits_{r \geq 1} \sum\limits_{m \geq 1}  s(m) \cdot h_1(m,n,r) \cdot 5^{\theta_1(m) + \pi_1(m,r)} \cdot x^r.
\end{align*}
For $r \geq 3$, Banerjee and Smoot prove
\begin{equation}\label{eq:bound}
\theta_1(m)+\pi_1(m,r) \geq \theta_0(r)+1.
\end{equation}
However, \eqref{eq:bound} is not true for $r=1$ with $1 \leq m \leq 4$ and for $r=2$ with $4 \leq m \leq 8$, as in these cases 
\[\theta_1(m)+\pi_1(m,r) - \theta_0(r) -1 = -1 < 0.\]
For these problematic values of $m$ and $r$, we see
\begin{align*} 
\sum\limits_{m=1}^5 s(m) \cdot h_1(m,n,1) &\equiv s(1) + s(2) + s(3) + 2s(4) + s(5) \pmod{5}, \\
\sum\limits_{m=4}^8 s(m) \cdot h_1(m,n,2) &\equiv 4s(4) + s(6) + s(7) + s(8) \pmod{5},
\end{align*}
thanks to Corollary \ref{cor:h_cor}. But $f \in \mathcal{V}_n^{(1)}$ implies that both of these sums vanish modulo $5$. Therefore, for $r \in \{1,2\}$ the missing power of 5 is accounted for in the coefficients $s(m) \cdot h_1(m,n,r),$ completing the proof.
\end{proof}
\end{thm}

So far, from Theorems \ref{thm:U0_Vhat} and \ref{thm:U1_V}, we know that 
\[\frac{1}{5} U^{(0)} \circ U^{(1)} : \mathcal{V}_n^{(1)} \rightarrow \hat{\mathcal{V}}_{25n+5}.\]
In order to complete the proof of Theorem \ref{thm:main2}, we need the image of this operator to lie in $\mathcal{V}_{25n+5}^{(0)}$ rather than just $\hat{\mathcal{V}}_{25n+5}.$ We establish this stability of $\frac{1}{5} U^{(0)} \circ U^{(1)}$ in the following theorem.

\begin{thm}\label{thm:stab}
Suppose $n \equiv 0 \pmod{5}$ and $f \in \mathcal{V}_n^{(1)}.$ Then 
\[\frac{1}{5} U^{(0)} \circ U^{(1)} (f) \in \mathcal{V}_{25n+5}^{(1)}.\]
\begin{proof}
Let
\[f = \frac{1}{(1+5x)^n} \sum\limits_{m \geq 1} s(m) \cdot 5^{\theta_1(m)} \cdot x^m,\]
so that by Theorem \ref{thm:U1_V}
\begin{align*} 
U^{(1)}(f) &= \sum\limits_{m \geq 1} s(m) \cdot 5^{\theta_1(m)} \cdot U^{(1)}\left(\frac{x^m}{(1+5x)^n}\right) \\
&= \frac{1}{(1+5x)^{5n}} \sum\limits_{m \geq 1}    \sum\limits_{r \geq \lceil m/5 \rceil} s(m) \cdot 5^{\theta_1(m)} \cdot h_1(m,n,r) \cdot 5^{\pi_1(m,r)} \cdot x^r \\
&=  \frac{1}{(1+5x)^{5n}} \sum\limits_{r \geq 1}    \sum\limits_{m \geq 1} s(m) \cdot 5^{\theta_1(m)+\pi_1(m,r)} \cdot h_1(m,n,r) \cdot x^r.
\end{align*}
Now from Theorem \ref{thm:U0_Vhat},
\begin{align*}
\frac{1}{5} U^{(0)} \circ U^{(1)} (f) &= \sum\limits_{r \geq 1} \sum\limits_{m \geq 1} s(m) \cdot 5^{\theta_1(m)+\pi_1(m,r)-1} \cdot h_1(m,n,r) \cdot U^{(0)} \left(\frac{x^r}{(1+5x)^{5n}}\right) \\ 
&=\sum\limits_{r \geq 1} \sum\limits_{m \geq 1} s(m) \cdot 5^{\theta_1(m)+\pi_1(m,r)-1} \cdot h_1(m,n,r) \cdot \frac{1}{(1+5x)^{25n+5}} \\
&\times \sum\limits_{w \geq \lceil (r+4)/5 \rceil} h_0(r,5n,w) \cdot 5^{\pi_0(r,w)} \cdot x^w \\
&= \frac{1}{(1+5x)^{25n+5}} \sum\limits_{w \geq 1} \sum\limits_{r \geq 1} \sum\limits_{m \geq 1} s(m) \cdot h_1(m,n,r) \cdot h_0(r,5n,w) \\
&\times 5^{\theta_1(m) + \pi_1(m,r) + \pi_0(r,w) -1} \cdot x^w
\end{align*}
Therefore, we can write
\begin{equation}
\frac{1}{5} U^{(0)} \circ U^{(1)} (f) = \frac{1}{(1+5x)^{25n+5}} \sum\limits_{w \geq 1} t(w) \cdot 5^{\theta_1(w)} \cdot x^w,
\end{equation}
where
\[t(w) \coloneqq \sum\limits_{r=1}^{5w-4} \sum\limits_{m=1}^{5r} s(m) \cdot h_1(m,n,r) \cdot h_0(r,5n,w) \cdot 5^{\theta_1(m) + \pi_1(m,r) + \pi_0(r,w) -\theta_1(w) -1}.\]
Thus, all that remains to be shown is that 
\begin{equation}\label{eq:t_congs}
\begin{pmatrix} t(1) + t(2) + t(3) + 2t(4) + t(5) \\ 4t(4) + t(6) + t(7) + t(8) \end{pmatrix} \equiv \begin{pmatrix} 0 \\ 0 \end{pmatrix} \pmod{5}.
\end{equation}
Clearly if $t(w) \equiv 0 \pmod{5}$ for some $1 \leq w \leq 8,$ then we may ignore it in the sum(s) appearing on the left-hand side of \eqref{eq:t_congs}. Therefore, when computing $t(w)$ for the purposes of verifying \eqref{eq:t_congs}, it suffices to only consider $r,m$ for which 
\begin{equation}\label{eq:ineq}
\theta_1(m)+\pi_1(m,r)+\pi_0(r,w)-\theta_1(w)-1 < 0.
\end{equation}
Hence, if we define 
\[\hat{t}(w) \coloneqq \sum\limits_{r=1}^{5w-4} \sum\limits_{\substack{m=1 \\ r,m \text{ satisfy } \eqref{eq:ineq}}}^{5r} s(m) \cdot h_1(m,0,r) \cdot h_0(r,0,w) \cdot 5^{\theta_1(m) + \pi_1(m,r) + \pi_0(r,w) -\theta_1(w) -1},\]
then by Proposition \ref{prop:h_cong} (remembering the hypothesis that $n \equiv 0 \pmod{5}$) and the above discussion, it suffices to show that 
\[\begin{pmatrix} \hat{t}(1) + \hat{t}(2) + \hat{t}(3) + 2\hat{t}(4) + \hat{t}(5) \\ 4\hat{t}(4) + \hat{t}(6) + \hat{t}(7) + \hat{t}(8) \end{pmatrix} \equiv \begin{pmatrix} 0 \\ 0 \end{pmatrix} \pmod{5}.\]
Computing these values of $\hat{t}(w)$, we find $\hat{t}(1)=0$ and 
\begin{align*}
\begin{split}
\hat{t}(2) &= \frac{1}{5} \left(4961 s(1) + 10406 s(2) + 6171 s(3) + 4575812 s(4) + 2991921 s(5) \right) \\
&+ 236628 s(6) + 58408 s(7) + 8848 s(8),
\end{split} \\
\begin{split} 
\hat{t}(3) &= \frac{1}{5}\left(388844s(1) + 815624 s(2) + 483684 s(3) + 1151708938 s(4) + 753195134 s(5)\right) \\
&+ 59571099 s(6) + 14704214 s(7) + 2227484 s(8), \end{split} \\
\begin{split}
\hat{t}(4) &= 2632405 s(1) + 5521630s(2) + 3274455s(3) + 18352874062s(4) \\
&+ 12003016215s(5) + 4746698523s(6) + 1171649878s(7) + 177488668s(8)   , \end{split} \\
\begin{split} 
\hat{t}(5) &= 51888370s(1) + 108839020s(2) + 64544070s(3) + 767043871640s(4) \\
&+ 501666289570s(5) + 198388915200s(6) + 48969267200s(7) + 7418163200s(8) ,
\end{split} \\
\begin{split} 
\hat{t}(6) &= 671367825s(1) + 1408234950s(2) + 835116075s(3) + 20258321552900s(4) \\
&+ 13249576649825s(5) + 5239683382500s(6) + 1293335645000s(7) + 195922370000s(8),
\end{split} \\
\begin{split}
\hat{t}(7) &=  6053383500s(1) + 12697341000s(2) + 7529818500s(3) + 371382334240250s(4) \\
&+ 242896720484750s(5) + 96056103747375s(6) + 23709978986750s(7) \\
&+ 3591732195500s(8),
\end{split} \\
\begin{split}
\hat{t}(8) &= 39202970000s(1) + 82230620000s(2) + 48764670000s(3) + 4989377344380000s(4) + \\
&+3263230087870000s(5) + 1290479615910000s(6) + 318535141260000s(7) \\
&+ 48253645560000s(8).
\end{split}
\end{align*}
Then
\begin{equation}\label{eq:t_hat1}
\begin{split}
\hat{t}(1)+\hat{t}(2) + \hat{t}(3) + 2\hat{t}(4) + \hat{t}(5) &= 57231941 s(1) + 120047486s(2) + 71190951s(3) \\
&+ 803980876714s(4) + 525823559411s(5) \\
&+ 207942119973s(6) + 5132732957s(7) \\
&+ 7775376868s(8)
\end{split}
\end{equation}
and 
\begin{equation}\label{eq:t_hat2}
\begin{split}
4\hat{t}(4) + \hat{t}(6) + \hat{t}(7) + \hat{t}(8) &= 45938250945s(1) + 96358282470s(2) + 57142702395s(3) \\
&+ 5381091411669398s(4) + 3519424397069435s(5) \\
&+ 1391794389833967s(6) + 343543142491262s(7) \\
&+ 52042010080172s(8).
\end{split}
\end{equation}
Let $I$ be the ideal of $\ZZ[s(1), \dots, s(8)]$ defined by 
\[I \coloneqq \left\langle s(1) + s(2) + s(3) + 2s(4) + s(5), 4s(4) + s(6) + s(7) + s(8)\right\rangle .\]
Using a computer algebra software, we see that the right-hand sides of \eqref{eq:t_hat1} and \eqref{eq:t_hat2} reduce to 0 modulo $5I$, therefore establishing \eqref{eq:t_congs} and completing the proof.  
\end{proof}
\end{thm}

We now complete the proof of Theorem \ref{thm:main2} by induction, showing that $\frac{1}{{5^{\lfloor \alpha/2\rfloor}}} \cdot L_{\alpha} \in \mathcal{V}_{\psi(\alpha)}^{(i)},$ where $i \equiv n \pmod{2}.$ Indeed, from \eqref{eq:L1} it is clear that $L_1 \in \hat{\mathcal{V}}_5.$ To verify that $L_1 \in \mathcal{V}_5^{(1)},$ we just need to check that the coefficients of the numerator in our rational polynomial expression for $L_1$ satisfy the defining congruences given in \eqref{eq:Vn1}. This is straightforward to see, as 
\begin{align*}
1 + 40 + 794 + 2\cdot 9125 + \frac{64475}{5}  &\equiv 0 \pmod{5}, \\
4\cdot 9125+ \frac{286000}{25} + \frac{7800000}{25} + \frac{800000}{125} &\equiv 0 \pmod{5}. 
\end{align*}
If $\frac{1}{{5^{\lfloor \alpha/2\rfloor}}} \cdot L_{\alpha} \in \mathcal{V}_{\psi(\alpha)}^{(i)}$ for some odd $\alpha \geq 1,$ then Theorems \ref{thm:U1_V} and \ref{thm:stab} immediately imply 
\[\frac{1}{5} L_{\alpha+1} = \frac{1}{5} U^{(1)}(L_\alpha) \in \mathcal{V}_{5\cdot \psi(\alpha)}^{(0)}\]
and
\[\frac{1}{5} L_{\alpha+2} = \frac{1}{5} U^{(0)} \circ U^{(1)} (L_\alpha) \in \mathcal{V}_{25\cdot \psi(\alpha)+5}^{(1)}.\]
Hence, we are done if we show $\psi(\alpha+1) = 5 \psi(\alpha)$ and $\psi(\alpha+2) = 25\psi(\alpha)+5$ for all odd $\alpha \geq 1.$ But this is immediate since, for $\alpha \geq 1$ odd, 
\[5\psi(\alpha) = 5\left\lfloor \frac{5^{\alpha+2}}{24} \right\rfloor= 5\left(\frac{5^{\alpha+2}}{24} - \frac{5}{24}\right) = \left(\frac{5^{\alpha+3}}{24} - \frac{1}{24}\right) -1 = \psi(\alpha+1)\]
and
\begin{align*}
25\psi(\alpha)+5 &= 25\left\lfloor \frac{5^{\alpha+2}}{24} \right\rfloor + 5 = 25\left(\frac{5^{\alpha+2}}{24} - \frac{5}{24}\right)+5 = \frac{5^{\alpha+4}}{24} - \frac{125}{24} + 5 \\
&= \frac{5^{\alpha+4}}{24} - \frac{5}{24} = \psi(\alpha+2).
\end{align*}

\section{Further Congruences and Closing Remarks}\label{sec:closing}
We close this work by briefly highlighting two additional congruences satisfied by $a_d(n)$. Here we assume familiarity with classical results on modular forms of integer weight. 

\begin{thm}\label{thm:further_congs}
For all $n \geq 0,$ we have 
\[a_5(49n+31) \equiv 0 \pmod{7}\]
and 
\[a_9(121n+36) \equiv 0 \pmod{11}.\]
\begin{proof}
Consider the functions 
\[F_1 \coloneqq \frac{\eta(\tau)^{440}}{\eta(2\tau)^4}, \quad F_2 \coloneqq \frac{\eta(\tau)^{2056}}{\eta(2\tau)^8}.\]
Using slight generalizations of Lemmas \ref{lem:eta-quotient1} and \ref{lem:eta-quotient2}, along with Lemma \ref{lem:Umap}, we see 
\[U_{49}(F_1) \in M_{218}(\Gamma_0(4)), \quad U_{121}(F_2)\in M_{1024}(\Gamma_0(2)).\]
Now notice
\[U_{49}(F_1) = U_{49}\left(\frac{q^{18}}{f_1 f_2^4} \cdot f_1^{441}\right) = f_1^9 \sum\limits_{n=0}^\infty a_5(49n+31) q^{n+1},\]
thanks to Lemma \ref{lem:U_PS}. Similarly, it follows that
\[U_{121}(F_2) = f_1^{17} \sum\limits_{n=0}^\infty a_9(121n+36)q^{n+1}.\]
Therefore, it suffices to show $U_{49}(F_1) \equiv 0 \pmod{7}$ and $U_{121}(F_2) \equiv 0 \pmod{11}.$ By a well-known theorem of Sturm (see, e.g., \cite{Sturm}), it is sufficient to check these congruences for only finitely many Fourier coefficients, those up to $q^{109}$ and $q^{256},$ respectively. This behavior is quickly verified with computer algebra software, such as Maple. 
\end{proof}
\end{thm}

We encourage the interested reader to further investigate arithmetic properties of generalized cubic partitions. In particular, we would be very interested to see if the congruences in Theorem \ref{thm:further_congs} fit into infinite families, akin to that of Theorem \ref{thm:main}.

\appendix
\section{Initial relations for the proof of Theorem \ref{thm:h's}}\label{app}
\begin{align*}
\begin{split}
U^{(0)}(1) &= \frac{1}{(1+5x)^5} (x+40x^2+794x^3+9125x^4+64475x^5+286000x^6 \\&+7800000x^7+1200000x^8+800000x^9),
\end{split} \\
\begin{split}
U^{(0)}(x) &= \frac{1}{(1+5x)^5}(121x^2+9484x^3+321025x^4+6327850x^5+81874125x^6+738217500x^7 \\ 
&+ 4780850000x^8+22488800000x^9+76460000000x^{10}+183600000000x^{11} \\
&+ 296000000000x^{12} + 288000000000x^{13} + 128000000000x^{14}) \end{split} \\
\begin{split} 
U^{(0)}(x^2) &= \frac{1}{(1+5x)^5} (140x^2 + 35245x^3 +2808365x^4 +117376000x^5 + 3100037500x^6 \\
&+ 56831205625x^7 + 763507050000x^8 + 7771895500000x^9 +61182640000000x^{10}\\
&+  376797500000000x^{11} +1823151200000000x^{12}  + 6913681600000000x^{13} \\
&+ 20347776000000000x^{14}+ 45594240000000000x{^15} + 75238400000000000x^{16}\\
&+ 86272000000000000x^{17} + 61440000000000000x^{18} + 20480000000000000x^{19})
\end{split}\\
\begin{split}
U^{(0)}(x^3) &= \frac{1}{(1+5x)^5}( 64 x^{2}+59136 x^{3}+10547620 x^{4}+850378650 x^{5}+40530512250 x^{6}\\
&+1298590915000x^7 +30103152240625 x^{8}+528858099450000 x^{9}\\
&+7262462532500000 x^{10}+79624221710000000x^{11} +707157357820000000 x^{12}\\
&+5136043622800000000 x^{13}+30674128864000000000 x^{14}\\
&+150938497280000000000x^{15} +611049315200000000000 x^{16}\\
&+2024736448000000000000 x^{17} +5439398912000000000000 x^{18}\\
&+11668643840000000000000x^{19} +19525324800000000000000 x^{20} \\
&+24567808000000000000000x^{21} +21872640000000000000000x^{22} \\
&+12288000000000000000000 x^{23} +3276800000000000000000 x^{24})  
\end{split} 
\end{align*}

\begin{align*}
\begin{split}
U^{(0)}(x^4) &= \frac{1}{(1+5x)^5}(13 x^{2}+54342 x^{3}+21645560 x^{4}+3231134475 x^{5}+261994052875 x^{6}\\
&+13648364390000 x^{7}+501535624578125 x^{8}+13781722427603125 x^{9}\\
&+294509461032250000 x^{10} +5030953041631500000 x^{11}+70082663702580000000 x^{12}\\
&+807847198383100000000 x^{13} +7788550590672000000000 x^{14}\\
&+63288334721120000000000 x^{15}+435690966505600000000000 x^{16}\\
&+2548759030153600000000000 x^{17}+12682694057728000000000000 x^{18} \\
&+53629451281920000000000000 x^{19} +192118376960000000000000000 x^{20}\\
&+579878492672000000000000000 x^{21}+1462436904960000000000000000 x^{22}\\
&+3044539924480000000000000000 x^{23}+5141902131200000000000000000 x^{24}\\
&+6869368832000000000000000000 x^{25}+6987448320000000000000000000 x^{26}\\
&+5085593600000000000000000000 x^{27}+2359296000000000000000000000 x^{28}\\
&+524288000000000000000000000 x^{29})    
\end{split}.
\end{align*}

\bibliographystyle{alpha}
\bibliography{refs}
\end{document}